\documentclass[reqno]{amsart}

\newtheorem{bigthm}{Theorem}
\newtheorem{theorem}{Theorem}[section]

\newtheorem{lemma}[theorem]{Lemma}

\newcommand{\eps}{\varepsilon}

\DeclareMathOperator{\li}{li}

\numberwithin{equation}{section}

\title{Additive bases arising from functions in a Hardy field}
\author{T.H. Chan} \address{Department of Mathematical Sciences\\
  University of Memphis\\ Memphis, TN 38152} \email{tchan@memphis.edu}

\author{A.V. Kumchev} \address{Department of Mathematics\\ 7800 York
  Road\\ Towson University\\ Towson, MD 21252}
\email{akumchev@towson.edu}

\author{M. Wierdl} \address{Department of Mathematical Sciences\\
  University of Memphis\\ Memphis, TN 38152}
\email{wierdlmate@gmail.com}

\thanks{The authors would like to thank
James Campbell, Nikos Frantzikinakis, Vladimir Nikiforov and
Andr\'as S\'ark\"ozy for helpful conversa\-tions.}

\begin{document}

\begin{abstract}
A classical additive basis question is Waring's problem. It has been
extended to integer polynomial and non-integer power sequences. In
this paper, we will consider a wider class of functions, namely
functions from a Hardy field, and show that they are asymptotic
bases.
\end{abstract}

\maketitle

\section{Introduction}

Let $k$ be a positive integer.  Waring's problem asks whether the
sequence $1^k,2^k,\dots$ of $k$th powers is an asymptotic basis.  In
other words, whether there are positive integers $s$ and $N_0$ such
that every integer $N$ greater than $N_0$ can be written in the form
\begin{equation}  \label{eq:1}
  n_1^k +n_2^k+\dots+n_s^k = N,
\end{equation}
with $n_1, n_2, \dots, n_s\in\mathbb N$. After Hilbert's affirmative
solution to this problem, several generalizations and extensions
were formulated, and important methods---such as the circle and
sieve methods, were developed to tackle those problems.

One variant of Waring's problem replaces the sequence of $k$th
powers by the range of an integer polynomial.  In other words, we
want to represent an integer $N$ in the form
\[
  f(n_1) + f(n_2) + \dots + f(n_s) = N,
\]
where $f(x) \in \mathbb Z[x]$. As a natural generalization of the
classical Waring problem, this question has been studied extensively
via the circle method (see Ford \cite{Ford99} for the history of the
problem and the best results to date).  In particular, it is known
that the sequence $f(n)$, $n=1,2,\dots$, is an asymptotic basis
whenever there is no integer $d \ge 2$ such that $d \mid
(f(n)-f(1))$ for all $n \in \mathbb N$.

Another generalization of Waring's problem is whether
\emph{non-integer} powers form a basis? It was Segal \cite{Segal}
who proved that for any fixed positive real number $c$, the sequence
$[1^c],[2^c],\dots$ of integer parts of $c$th powers is an
asymptotic basis. This question has been studied further by
Deshouillers \cite{Desh74} and by Arkhipov and Zhitkov
\cite{ArZh84}, among others. Whereas those authors focused on the
order of the basis (i.e., the number of unknowns in an equation
analogous to \eqref{eq:1} that ensure solvability), in this work we
ask the general question if the sequence $[n^c]$ can be replaced by
other sequences.  What other sequences do we have in mind?  Well,
for one, take a polynomial $p(x)$ with an irrational coefficient of
a non-constant term.  Is then $[p(n)]$, $n=1,2,\dots$ a basis? How
about more general sequences such as
\begin{align}  \label{eq:2}
  \left[\frac{n^2}{\log n}\right],
  \quad \text{or} \quad
  \bigg[\pi n^3 + \frac{n^{\sqrt 2}}{\log\log n}\bigg].
\end{align}
In other words, we are interested in sequences of the form
$[f(1)],[f(2)],\dots $ when $f$ is a smooth positive function such
that $f(x) \to \infty$ as $x \to \infty$. An obvious necessary
condition is that $f$ grows no faster than a polynomial:
\begin{equation}\label{1.1}
  f(x) \ll (|x| + 1)^k \quad \text{for some } k \in \mathbb N.
\end{equation}
For the rest of this note, unless we say otherwise, all functions
are assumed to satisfy this growth condition.

As we mentioned, the case of rational polynomials has been
extensively studied, so in this introduction we consider functions
$f(x)$ which are far away from rational polynomials in the sense
that
\begin{equation}
  \label{eq:3}
  |f(x)-g(x)| \to \infty \text{ for every } g(x)\in \mathbb Q[x].
\end{equation}

Our aim is a general theorem which will include Segal's theorem as a
special case, but it also includes the case of irrational
polynomials and is {easily} and readily applicable to somewhat wild
sequences such as those in \eqref{eq:2}.  Besides the growth
condition \eqref{1.1}, we need some regularity condition on $f$.
Instead of trying to describe various growth conditions on the
derivatives of the function $f$, we seek conditions which are easily
applicable to a wide class of functions, including those appearing
in \eqref{eq:2}.  As an illustration of the type of functions we
want to consider, take the so called logarithmico exponential
functions of Hardy. These are functions we obtain by writing down a
``formula'' using the symbols $\log,\exp,+,\cdot,x,c$ where $x$ is a
real variable and $c$ is a real constant.  This class of functions
certainly includes any power function since $x^c=\exp(c\log x)$.  It
similarly includes any polynomials.  Now our theorem will imply that
if $f(x)$ is a logarithmico exponential function that satisfies
\eqref{1.1} and \eqref{eq:3} then the sequence $[f(1)],[f(2)], \dots
$ forms a basis. We can see immediately that the sequences in
\eqref{eq:2} form a basis.  A small technical remark is that the
function $f$ may be defined only for large enough $x$ (for example,
consider $x^2\log\log\log x$).  In that case, we agree that we
consider, instead, $f(x+k)$ with a suitably large, fixed $k$.

While the logarithmico exponential functions already represent a
large class of functions, sequences involving the logarithmic
integral $\li x$ or $\Gamma$-functions are still not covered; we
want to admit functions such as
\begin{equation}
  \label{eq:4}
  (\li x)^2 \text{ or } (\log \Gamma(x))^{\sqrt 2},
\end{equation}
and even products, quotients, sums or differences of these
functions.

The functions we want to consider are those from a {Hardy field}.
For a more extensive introduction to Hardy fields and for further
references about the facts we claim below, see Boshernitzan
\cite{Bosh94}. To define Hardy fields, we consider first the ring of
continuous functions with pointwise addition and multiplication as
the ring operations.  Since functions such as $\log x$, $\log\log
x$, and $\log\log\log x$ are defined only on some neighborhood of
$\infty$, we want to consider \emph{germs} of functions, that is,
equivalence classes of functions, where we consider two functions
equivalent if they are equal in a neighborhood of $\infty$.  Let $B$
the ring of all these germs of continuous functions. A \emph{Hardy
field} is a subring of $B$ which is a {field}, and which is also
closed with respect to differentiation. It is known that $E$, the
intersection of all {maximal} Hardy fields, contains all
logarithmico exponential functions and is closed under
antidifferentiation; hence, $\li x \in E$. Another known fact is
that there is a Hardy field containing the $\Gamma$-function, and
that the intersection of all maximal translation invariant Hardy
fields contains $\log \Gamma(x)$.

We remark that by virtue of their very definition, functions in a
Hardy field are defined only on a neighborhood of $\infty$. On the
other hand, given a particular function in a Hardy field, we want to
think of it as being defined on $[1, \infty)$. To fix this, we
replace $f(x)$ by $f(x + k)$ for some sufficiently large constant
$k$. With this caveat, our main result is the following theorem.

\begin{bigthm}
  \label{thmA}
  Suppose that $f$ is a function from a Hardy field that
  satisfies the growth condition \eqref{1.1} and the condition
  \[
  \lim_{x \to \infty} \big| f(x) - g(x) \big| = \infty \qquad
  \text{for all $g(x) \in \mathbb Q[x]$}. \tag{\ref{eq:3}}
  \]
  Then the sequence $[f(n)]$, $n=1,2,\dots$, is an asymptotic basis.
\end{bigthm}

The utility of the formulation of our theorem in terms of Hardy
fields is that it is \emph{easily} applicable. Indeed, by the
preceding remarks, the sequences in \eqref{eq:2} all form a basis as
well as sequences such as $[(\li n)^2]$ or $\big[(\log
  \Gamma(n))^{\sqrt 2}\big]$.

\section{Preliminaries}

\paragraph{\bf Notation.} Most of our notation is standard. We write
$e(x) = e^{2\pi ix}$. For a real number $\theta$, $[\theta]$,
$\{\theta\}$ and $\| \theta \|$ denote, respectively, the integer
part of $\theta$, the fractional part of $\theta$ and the distance
from $\theta$ to the nearest integer. Also, we use Vinogradov's
notation $A \ll B$ and Landau's big-oh notation $A = O(B)$ to
indicate that $|A| \le KB$ for some constant $K > 0$.

\subsection{Hardy fields}

Let $\mathfrak U$ denote the union of all Hardy fields. Suppose that
$f \in \mathfrak U$ satisfies condition \eqref{1.1} and that $f(x)
\to \infty$ as $x \to \infty$. We say that a function $f \in
\mathfrak U$ is \emph{non-polynomial} if for every $k \in \mathbb
N$, we have $f(x) = o(x^k)$ or $x^k = o(f(x))$ as $x \to \infty$.
Each function $f \in \mathfrak U$ that satisfies the growth
condition \eqref{1.1} falls in one of the following three classes
(see \cite{MR2145587}):
\begin{itemize}
\item [i)] $f(x) = p(x)$, where $p(x) \in \mathbb R[x]$;
\item [ii)] $f(x) = r(x)$, where $r$ is non-polynomial;
\item [iii)] $f(x) = p(x) + r(x)$, where $p(x) \in \mathbb R[x]$ and
  $r$ is non-polynomial, with $r(x) = o(p(x))$ as $x \to \infty$.
\end{itemize}
Our proof of Theorem~\ref{thmA} will be given in two parts each
utilizing different methods.  The first method handles the case when
$f$ is far away from all polynomials, and the second handles the
rest of the cases.  More precisely,the first part of the proof is
given in Section~\ref{s3}, and it handles the case when $f$ either
belongs to class ii) above or it belongs to class iii) but for some
positive $\delta$ we have $r(x) \gg x^\delta$.  The second part of
the proof is given in Section~\ref{s5} it handles the rest of the
functions, so when either $f$ is a polynomial (class i)), or when
$f$ belongs to class iii) and $r(x) \ll x^\delta$ for all positive
$\delta$.

When $f$ belongs to the classes i) or iii), these assumptions mean
that the polynomial part of $f$ has a positive degree $d$. For
functions of class iii), we call $d$ the \emph{degree of $f$}. When
$f$ is non-polynomial (so class ii)), then there exists a real
number $c \ge 0$ such that, for any fixed $\eps > 0$, one has
\begin{equation}\label{2.1}
  x^{c - j - \eps} \ll f^{(j)}(x) \ll x^{c - j + \eps} \qquad (j = 0, 1, 2, \dots),
\end{equation}
the implied constants depending at most on $f, j$ and $\eps$. (See
 \cite{MR2145587} for a proof.) For functions of
class ii), we call the number $c$ in \eqref{2.1} the \emph{degree of
  $f$}. Thus, we have now defined the degree of any function $f \in
\mathfrak U$ subject to \eqref{1.1}. We denote the degree of $f$ by
$d_f$. By \eqref{2.1},
\begin{equation}\label{2.2}
  x^{d_f - j - \eps} \ll f^{(j)}(x) \ll x^{d_f - j + \eps} \qquad (j = 0, 1, \dots, d_f)
\end{equation}
for any function $f \in \mathfrak U$ satisfying the above
hypotheses.

When $f$ is of classes ii) or iii), we denote by $c_f$ the degree of
its non-polynomial part. Thus, $c_f = d_f$ or $c_f = d_r$ according
as $f$ is of class ii) or iii). In particular, we have
\begin{equation}\label{2.3}
  x^{c_f - j - \eps} \ll f^{(j)}(x) \ll x^{c_f - j + \eps} \qquad (j > d_f).
\end{equation}

We also recall the following result of
Boshernitzan~\cite[Theorem~1.4]{Bosh94}.

\begin{lemma}\label{l0}
  Suppose that $f \in \mathfrak U$ satisfies condition
  \eqref{1.1}. Then the following two conditions are equivalent.
  \begin{itemize}
  \item [(a)] The sequence $\big( \{f(n)\} \big)_{n \in \mathbb N}$ is
    dense in $[0, 1)$.
  \item [(b)] For every polynomial $g(x) \in \mathbb Q[x]$, one has
    \[
    \lim_{x \to \infty} \big| f(x) - g(x) \big| = \infty.
    \]
  \end{itemize}
\end{lemma}

\subsection{Bounds on exponential sums}

\begin{lemma}\label{l1}
  Let $k \ge 2$ be an integer and put $K = 2^k$. Suppose that $a \le b
  \le a + N$ and that $f:[a, b] \to \mathbb R$ has continuous $k$th
  derivative that satisfies the inequality $0 < \lambda \le
  |f^{(k)}(x)| \le h\lambda$ for all $x \in [a, b]$. Then
  \[
  \sum_{a \le n \le b} e (f(n)) \ll hN\big( \lambda^{1/(K - 2)} +
  N^{-2/K} + (N^k\lambda)^{-2/K} \big).
  \]
\end{lemma}

\begin{proof}
  This is a variant of van der Corput \cite[Satz 4]{vdCo29}.
\end{proof}

\subsection{The Hilbert--Kamke problem}
\label{s2.4}

Let $N_1, \dots, N_k$ be large positive integers. The Hilbert--Kamke
problem is concerned with the system of Diophantine equations
\[
x_1^j + x_2^j + \dots + x_s^j = N_j \quad (1 \le j \le k).
\]
It is known that this system has solutions in positive integers
$x_1, \dots, x_s$, provided that:
\begin{itemize}
\item [(a)] $s$ is sufficiently large;
\item [(b)] there exist real numbers $\mu_1, \dots, \mu_{k-1} > 1$ and
  $\nu_1, \dots, \nu_{k-1} < 1$ such that
  \[
  \mu_j N_k^{j/k} \leq N_j \leq \nu_j s^{1 - j/k} N_k^{j/k} \quad (1
  \leq j < k);
  \]
\item [(c)] the $N_j$'s satisfy the congruences
  \[
  \Delta_j(N_1, \dots, N_k) \equiv 0 \pmod {\Delta_0} \quad (1 \le
  j \le k),
  \]
  where
  \begin{equation}\label{2.4}
    \Delta_0 = \begin{vmatrix} 1 & 2 & \cdots & k \\ 1^2 & 2^2 & \cdots & k^2 \\ \vdots & \vdots & & \vdots \\ 1^k & 2^k & \cdots & k^k \end{vmatrix} = 1!2! \cdots k!
  \end{equation}
  and $\Delta_j(N_1, \dots, N_k)$ is the determinant resulting from
  replacing the numbers $j, \dots, j^k$ in the $j$th column of
  $\Delta_0$ by $N_1, \dots, N_k$, respectively.
\end{itemize}
The reader can find a proof of the sufficiency of conditions
(a)--(c) above in \cite[\S2.7]{GeLi66}, for example.

\subsection{Bounded gaps imply basis}
\label{s2.3}

For a sequence $\mathcal A = (a_n)_{n \in \mathbb N}$, we define the
\emph{sumset} $s\mathcal A$ by
\[
s\mathcal A = \big\{ a_1 + a_2 + \dots + a_s \mid a_i \in \mathcal A
\big\}.
\]
In the proof, we shall need the following elementary result.

\begin{lemma}\label{l3}
  Let $\mathcal A = (a_n)_{n \in \mathbb N}$ be an integer sequence
  such that:
  \begin{itemize}
  \item [(a)] for some $s \in \mathbb N$, the sumset $s\mathcal A$ has
    bounded gaps;
  \item [(b)] $\gcd\big\{ a_n - a_1 \mid n \in \mathbb N \big\} = 1$.
  \end{itemize}
  Then $\mathcal A$ is an asymptotic basis.
\end{lemma}

This lemma can be derived from more general results by Erd\H os and
Graham \cite[Theorem 1]{ErGr80} or by Nash and Nathanson \cite[Lemma
1]{NaNa85}. For the sake of completeness, we present a direct proof.

\begin{proof}
  By hypothesis (b), we have
  \[
  \gcd(a_2 - a_1, a_3 - a_1, \dots, a_k - a_1) = 1
  \]
  for some $k \in \mathbb N$. Thus, there exist integers $x_2, \dots,
  x_k$ such that
  \begin{equation}\label{2.6}
    \sum_{j = 2}^k x_j(a_j - a_1) = 1.
  \end{equation}
  Define the integers $a_j'$ and $a_j''$, $2 \le j \le k$, by
  \[
  a_j' = \begin{cases} a_j & \text{if } x_j \ge 0, \\ a_1 & \text{if }
    x_j < 0; \end{cases} \qquad a_j'' = \begin{cases} a_1 & \text{if }
    x_j \ge 0, \\ a_j & \text{if } x_j < 0. \end{cases}
  \]
  We can rewrite \eqref{2.6} as
  \begin{equation}\label{2.7}
    \sum_{j = 2}^k |x_j|a_j' = 1 + \sum_{j = 2}^k |x_j|a_j''.
  \end{equation}

  Let $g$ and $M$ be, respectively, the largest gap of $s\mathcal A$
  and the least element of $s\mathcal A$, and suppose that $N$ is an
  integer with
  \[
  N \ge M + g\sum_{j = 1}^k |x_j|a_j''.
  \]
  Then
  \[
  N - g\sum_{j = 1}^k |x_j|a_j'' = b + h
  \]
  for some integers $b$ and $h$, with $b \in s\mathcal A$ and $0 \le h
  \le g$. Hence, by \eqref{2.7},
  \begin{align*}
    N = b + h + g\sum_{j = 1}^k |x_j|a_j'' = b + h\sum_{j = 1}^k
    |x_j|a_j' + (g - h)\sum_{j = 1}^k |x_j|a_j''.
  \end{align*}
  This establishes that every sufficiently large $N$ is the sum of $s
  + g\sum_j|x_j|$ elements of $\mathcal A$. Thus, $\mathcal A$ is an
  asymptotic basis.
\end{proof}

\section{Proof of Theorem \ref{thmA} in case $r(x) \gg x^\delta$ }
\label{s3}

Let $\delta > 0$ and consider a function $f \in \mathfrak U$ with
its non polynomial part $r(x)$ satisfying $r(x) \gg x^\delta$.  For
simplicity, we write $d = d_f$ and $c = c_f$. In the case when $f$
is of class ii), we assume that $d \ge 1$; otherwise, the sequence
$a_n = [f(n)]$ contains all sufficiently large integers, and the
result is trivial.

\subsection{A variant of the circle method}
\label{s3.1}

Let $s \in \mathbb N$ and suppose that $N \ge N_0(s, d, \delta)$,
where $\delta$ is the positive number from the hypotheses of the
theorem. We set $X = N^{1/d}$ and $N_s = N/(s + 1)$, and we choose
$X_0$ and $X_1$ so that
\[
f(X_0) = N_s, \quad f(X_1) = 2N_s.
\]
Let $R_s(N)$ denote the number of solutions of the equation
\[
[f(n_1)] + [f(n_2)] + \dots + [f(n_s)] = N
\]
in integers $n_1, n_2, \dots, n_s$ with $X_0 < n_i \le X_1$. Then
\begin{equation}\label{3.1}
  R_s(N) = \int_{-1/2}^{1/2} S(\alpha)^s e(-\alpha N) \, d\alpha,
\end{equation}
where
\[
S(\alpha) = \sum_{X_0 < n \le X_1} e (\alpha[f(n)]).
\]
Put $\omega = X^{1/2 - d}$. We will show that when $s \ge 3$ and $0
< \eps < (6s)^{-1}$, one has
\begin{gather}
  \int_{-\omega}^{\omega} S(\alpha)^se(-\alpha N) \, d\alpha \gg X^{s - d - 2\eps}; \label{3.2}\\
  \sup_{\omega \le |\alpha| \le 1/2} |S(\alpha)| \ll X^{1 - \sigma +
    \eps}, \label{3.3}
\end{gather}
where $\sigma = \sigma(c, d) > 0$. Clearly, the theorem follows from
\eqref{3.1}--\eqref{3.3}.

Suppose that $|\alpha| \le \omega$ and define
\[
T(\alpha) = \sum_{X_0 < n \le X_1} e(\alpha f(n)), \quad I(\alpha) =
\int_{X_0}^{X_1} e(\alpha f(t)) \, dt.
\]
We have
\begin{equation}\label{3.4}
  S(\alpha) = T(\alpha) + O(\omega X_1).
\end{equation}
Furthermore, since
\[
\sup_{X_0 \le t \le X_1} |\alpha f'(t)| \ll \omega X^{d - 1 + \eps}
< 1/2,
\]
\cite[Lemma 8.8]{IwKo04} gives
\begin{equation}\label{3.5}
  T(\alpha) = I(\alpha) + O(1).
\end{equation}
Let $\Delta_1 = \omega X_1 + 1$. Combining \eqref{3.4} and
\eqref{3.5}, we find that
\begin{equation}\label{3.6}
  \int_{-\omega}^{\omega} \big| S(\alpha)^s - I(\alpha)^s \big| \, d\alpha \ll \Delta_1 \int_{-\omega}^{\omega} |I(\alpha)|^{s - 1} \, d\alpha + \omega\Delta_1^s.
\end{equation}
Since
\[
\inf_{X_0 \le t \le X_1} |\alpha f'(t)| \gg |\alpha| X^{d - 1 -
\eps},
\]
we deduce from \cite[Lemma 8.10]{IwKo04} that
\[
I(\alpha) \ll |\alpha|^{-1}X^{1 - d + \eps}.
\]
From the last inequality and the trivial bound for $I(\alpha)$, we
obtain
\begin{equation}\label{3.7}
  I(\alpha) \ll \frac {X^{1 + \eps}}{1 + X^d|\alpha|}.
\end{equation}
Hence, for $s \ge 3$,
\[
\int_{-\omega}^{\omega} |I(\alpha)|^{s - 1} \, d\alpha \ll
\int_{-\omega}^{\omega} \frac {X^{s - 1 + (s - 1)\eps} \,
d\alpha}{(1
  + X^d|\alpha|)^{s - 1}} \ll X^{s - d - 1 + (s - 1)\eps}.
\]
Upon noting that
\[
\Delta_1 \ll X^{3/2 - d + \eps} \ll X^{1/2 + \eps},
\]
we deduce from \eqref{3.6} that
\begin{equation}\label{3.8}
  \int_{-\omega}^{\omega} \big| S(\alpha)^s - I(\alpha)^s \big| \, d\alpha \ll X^{s - d - 1/3}.
\end{equation}

We now evaluate
\[
\int_{-\omega}^{\omega} I(\alpha)^s e(-\alpha N) \, d\alpha.
\]
By \eqref{3.7},
\[
\int_{|\alpha|>\omega} |I(\alpha)|^s \, d\alpha \ll
\int_{\omega}^{\infty} \frac {X^{s + s\eps} \, d\alpha}{(1 + \alpha
  X^{d})^s} \ll X^{s - d - 1/3},
\]
so
\begin{equation}\label{3.9}
  \int_{-\omega}^{\omega} I(\alpha)^s e(-\alpha N) \, d\alpha = \int_{\mathbb R} I(\alpha)^s e(-\alpha N) \, d\alpha + O\left( X^{s - d - 1/3} \right).
\end{equation}
If $g$ is the inverse function to $f$ on the interval $X_0 \le t \le
X_1$, then
\[
I(\alpha) = \int_{N_s}^{2N_s} g'(u) e(\alpha u) \, du = V(\alpha),
\quad \text{say}.
\]
By Fourier's inversion formula, the integral on the right side of
\eqref{3.9} equals
\[
\int_{\mathcal D_s} g'(u_1) \cdots g'(u_{s - 1}) g'(N - u_1 - \dots
- u_{s - 1}) \, d\mathbf u,
\]
where $\mathcal D_s$ is the $(s - 1)$-dimensional region defined by
the inequalities
\[
N_s \le u_1, \dots, u_{s - 1} \le 2N_s, \quad N_s \le N - u_1- \dots
- u_{s - 1} \le 2N_s.
\]
Note that $\mathcal D_s$ contains the $(s - 1)$-dimensional box
\[
N/(s + 1) \le u_1, \dots, u_{s - 1} \le N/s
\]
and
\[
\inf_{N_s \le u \le 2N_s} g'(u) \gg X_1^{1 - d - \eps/s}.
\]
We deduce that
\begin{equation}\label{3.10}
  \int_{\mathbb R} I(\alpha)^s e(-\alpha N) \, d\alpha \gg N^{s - 1}X_1^{s - sd - \eps} \gg X^{s - d - 2\eps}.
\end{equation}
Combining \eqref{3.8}--\eqref{3.10}, we obtain \eqref{3.2}.

We now proceed to the estimation of $S(\alpha)$ on the two minor
arcs. For non-integer reals $x, \alpha$ and $K \ge 2$, we have
\begin{equation}\label{3.11}
  e(-\alpha\{x\}) = c(\alpha) \sum_{|k| \le K} \frac {e(kx)}{k + \alpha} + O \left( \Phi(x; K) \log K \right),
\end{equation}
where $|c(\alpha)| \le \|\alpha\|$ and $\Phi(x; K) = (1 +
K\|x\|)^{-1}$. Furthermore,
\begin{equation}\label{3.12}
  \Phi(x; K) = \sum_{k \in \mathbb Z} b_ke(kx), \qquad |b_k| \ll \frac {K\log K}{K^2 + |k|^2}.
\end{equation}
By \eqref{3.11},
\begin{align}\label{3.13}
  S(\alpha) &= \sum_{X_0 < n \le X_1} e(\alpha f(n))e(-\alpha\{f(n)\}) \notag\\
  &= \sum_{|k| \le K} \frac {|c(\alpha)|}{|k + \alpha|} |T(k +
  \alpha)| + O( \Delta(K)\log K ),
\end{align}
where
\begin{equation}\label{3.14}
  \Delta(K) = \sum_{X_0 < n \le X_1} \Phi(f(n);K) \ll  \sum_{k \in \mathbb Z} |b_k||T(k)|.
\end{equation}
Combining \eqref{3.12}--\eqref{3.14}, we obtain
\begin{equation}\label{3.15}
  \sup_{\omega \le |\alpha| \le 1/2} |S(\alpha)| \ll \Big( \sup_{\omega \le |\beta| \le K^2}|T(\beta)| + X_1K^{-1} \Big)\log^2K.
\end{equation}
The estimate \eqref{3.3} follows readily from \eqref{3.15} with $K =
X^{\sigma}$ and the bound
\begin{equation}\label{3.16}
  \sup_{\omega \le |\beta| \le X^{2\sigma}} |T(\beta)| \ll X^{1 - \sigma},
\end{equation}
which we establish in the next section.

\subsection{Estimation of $T(\beta)$}
\label{s3.2}

We now establish \eqref{3.16}. By a standard dyadic argument, we can
reduce \eqref{3.16} to the estimation of the exponential sum
\begin{equation}\label{3.17} W(\beta) = \sum_{P < n \le P_1} e
  \big( \beta f(n) \big),
\end{equation}
where $X_0 \le P < P_1 \le \min(2P, X_1)$ and $P^{1/2 - d - \eps}
\le |\beta| \le P^{2\sigma + \eps}$. We also assume, as we may, that
$0 < \sigma < \frac 14$.

Suppose first that $f$ is of class iii) and set $\eta = \frac 14
\min(1, c)$. We consider two cases depending on the relative sizes
of $\beta$ and $P$.

\paragraph*{{\bf Case 1:}} $P^{1/2 - d - \eps} \le |\beta| \le P^{-c +
  \eta}$. Let
\[
p(x) = \alpha x^d + \cdots + \alpha_kx^k
\]
be the polynomial part of $f$. When $d \ge 2$, we apply Lemma
\ref{l1} with $k = d$, $N = P$, $\lambda = |\beta|$, and $h =
P^{\eps}$. We obtain
\[
W(\beta) \ll P^{1 + \eps} \big( |\beta|^{1/(K - 2)} + P^{-2/K} +
(|\beta|P^d)^{-2/K} \big) \ll P^{1 - \sigma_1 + \eps},
\]
where $K = 2^d$ and $\sigma_1 = \min( 1, c/2 )K^{-1}$. When $d = 1$,
we have
\[
|\beta|P^{-\eps} \ll |\beta f'(x)| \ll |\beta|P^{\eps},
\]
so the Kuzmin--Landau inequality (see \cite[Corollary 8.11]{IwKo04})
gives
\[
W(\beta) \ll P^{\eps}|\beta|^{-1} \ll P^{1/2 + 2\eps}.
\]

\paragraph*{{\bf Case 2:}} $P^{-c + \eta} \le |\beta| \le P^{2\sigma +
  \eps}$. We recall \eqref{2.3} and apply Lemma \ref{l1} with $k = d +
1$, $N = P$, $\lambda = |\beta|P^{c - k - \eps/4}$, and $h =
P^{\eps/2}$. We get
\[
W(\beta) \ll P^{1 + \eps} \big( (|\beta|P^{c - k})^{1/(J - 2)} +
P^{-2/J} + (|\beta|P^c)^{-2/J} \big),
\]
where $J = 2^{d + 1}$. Since
\[
(|\beta|P^{c - k})^{1/(J - 2)} \ll P^{(2\sigma - 1 + \eps)/(J - 2)}
\ll P^{-1/(2J)},
\]
we deduce that
\[
W(\beta) \ll P^{1 + \eps} \big( P^{-1/(2J)} + P^{-2\eta/J} \big).
\]

Combining the above estimates, we conclude that when $f$ is of class
iii), \eqref{3.16} holds with $\sigma = 2^{-d - 2}\min(c, 1)$. When
$f$ is of class ii), we can argue similarly to Case 2 above to show
that \eqref{3.16} holds with $\sigma = 2^{-k - 1}$, where $k =
\lceil d + 1 \rceil$. \qed

\section{Proof of Theorem \ref{thmA} when $r(x) \ll
  x^\delta$ for all $\delta$}
\label{s5}

In this section, we assume that $f$ is either of class i) or it is
of class iii) with $r(x) \ll x^\delta$ for all $\delta$. In this
case, Theorem~\ref{thmA} follows from Lemma \ref{l3}. Under the
hypotheses of Theorem~\ref{thmA}, we have
\[
\lim_{x \to \infty} \big| q^{-1}f(x) - g(x) \big| = \infty
\]
whenever $q \in \mathbb N$ and $g(x) \in \mathbb Q[x]$. Hence, by
Lemma \ref{l0}, the fractional parts $\big\{ q^{-1}f(n) \big\}$ are
dense in $[0, 1)$. In particular, the Diophantine inequality
\[
\frac {a - 1}q \le \bigg\{ \frac {f(n)}q \bigg\} < \frac {a}q.
\]
has solutions for any given integers $a$ and $q$, with $1 \le a \le
q$. Therefore, every arithmetic progression $a \pmod q$ contains an
element of $\mathcal A$. This establishes that $\mathcal A$
satisfies hypothesis (b) of Lemma \ref{l3}.

We see, that it is enough to show that the set of sums $[f(x_1)] +
\dots + [f(x_s)]$ has bounded gaps when $s$ is sufficiently large.
It suffices to show that every large real $N$ lies within a bounded
distance from a sum $f(x_1) + \dots + f(x_s)$. Suppose that $f(x) =
p(x) + r(x)$, where
\[
p(x) = \alpha_kx^k + \dots + \alpha_1x, \quad r(x) \ll (|x| +
1)^{\delta},
\]
with $0 < \delta < 1/2$. We define $U$ and $V$ in terms of $N$ by
the equations
\[
N = sf(U + V) + s\alpha_kV^k, \quad V = U^{1 - \delta}.
\]
These are well-defined, since the function $f(U + V) + \alpha_kV^k$
is strictly increasing for $U$ sufficiently large. We then set $X =
[U]$ and $Y = V + \{ U \}$, so that $X + Y = U + V$. Using the
Taylor expansion
\[
f(X + Y) = f(X) + \sum_{j = 1}^k \frac {f^{(j)}(X)}{j!}Y^j +
O(X^{-\delta}),
\]
we obtain
\begin{equation}\label{5.1}
  N = sf(X) + \sum_{j = 1}^{k - 1} \frac {f^{(j)}(X)}{j!}(sY^j) + s\alpha_k(Y^k + V^k) + O(X^{-\delta}).
\end{equation}
We are going to use the result on the Hilbert--Kamke problem to
replace the terms $sY^j$ in \eqref{5.1} by $s$-fold sums of $j$th
powers. Let $\Delta_0$ denote the determinant in \eqref{2.4}. The
idea is to approximate each $sY^j$ by a suitable multiple of
$\Delta_0$ and carry the residual error to the next step. We first
find a positive integer $M_1$ such that $0 < sY^1 - \Delta_0 M_1
\leq \Delta_0$. Let $E_1 = sY - \Delta_0 M_1$ be the residual error.
Then \eqref{5.1} becomes
\begin{align*}
  N = {}& sf(X) + \frac{f'(X)}{1!} \Delta_0 M_1 + \frac{f'(X)}{1!} E_1 + \frac {f''(X)}{2!}(sY^2) + \cdots \\
  & + s \alpha_k(Y^k + V^k) + O(X^{-\delta}) \\
  ={}& sf(X) + \frac{f'(X)}{1!} \Delta_0 M_1 + \frac {f''(X)}{2!} \left( sY^2 + \frac{2f'(X)}{f''(X)} E_1 \right) + \cdots \\
  & + s\alpha_k(Y^k + V^k) + O(X^{-\delta}).
\end{align*}
Next, we find a positive integer $M_2$ such that
\[
0 < sY^2 + \frac{2f'(X)}{f''(X)} E_1 - \Delta_0 M_2 \leq \Delta_0.
\]
We then carry the residual error
\[
E_2 = s Y^2 + \frac{2f'(X)}{f''(X)} E_1 - \Delta_0 M_2
\]
to next step and repeat the process. Upon setting $E_0 = 0$, this
process yields a recursive integer sequence $M_1, M_2, \dots, M_k$,
defined by the conditions
\begin{gather*}
  0 < E_j = sY^j + \frac {j f^{(j - 1)}(X)}{f^{(j)}(X)}E_{j - 1} - \Delta_0 M_j \le \Delta_0 \quad (1 \le j < k); \\
  0 < E_k = s(Y^k + V^k) + \frac {k f^{(k - 1)}(X)}{f^{(k)}(X)}E_{k -
    1} - \Delta_0 M_k \le \Delta_0.
\end{gather*}
Substituting into \eqref{5.1}, we get
\begin{equation}\label{5.2}
  N = sf(X) + \sum_{j = 1}^{k} \frac {f^{(j)}(X)}{j!} \Delta_0 M_j + O(1),
\end{equation}
By the choices of the $M_j$'s, we have
\begin{align*}
  \Delta_0 M_j &= sY^j(1 + O(XY^{-2})) \quad (1 \le j < k), \\
  \Delta_0 M_k &= 2sY^k(1 + O(XY^{-2})),
\end{align*}
whence
\[
\Delta_0 M_j = 2^{-j/k}s^{1 - j/k} (\Delta_0 M_k)^{j/k}(1 +
O(XY^{-2})) \quad (1 \le j < k).
\]
Recalling that $Y \gg X^{1 - \delta}$ and $\delta < 1/2$, we
conclude that the integers $\Delta_0 M_1, \dots, \Delta_0 M_k$
satisfy conditions (b) and (c) in \S\ref{s2.4}. Therefore, if $s$ is
sufficiently large, there exist positive integers $y_1, \dots, y_s$
such that
\[
\Delta_0 M_j = y_1^j + \dots + y_s^j \quad (1 \le j \le k).
\]
Substituting these into \eqref{5.2}, we obtain
\begin{align*}
  N &= sf(X) + \sum_{j = 1}^{k} \frac {f^{(j)}(X)}{j!}(y_1^j + \dots + y_s^j) + O(1) \\
  &= f(X + y_1) + \dots + f(X + y_s) + O(1).
\end{align*}
Thus, as desired, $N$ lies within a bounded distance of a sum of the
form $f(x_1) + \dots + f(x_s)$. \qed

\bibliographystyle{amsplain}

\end{document}